\documentclass[a4paper,12pt, reqno]{amsart}
\usepackage{amssymb}
\usepackage{cite}
\usepackage{ifthen}
\usepackage[dvips]{graphicx}
\nonstopmode \numberwithin{equation}{section}
\newtheorem*{theoA}{Theorem A}
\newtheorem*{theoB}{Theorem B}
\newtheorem*{theoC}{Theorem C}
\newtheorem*{theoD}{Theorem D}
\newtheorem*{theoE}{Theorem E}
\newtheorem*{theoF}{Theorem F}

\usepackage{amssymb}
\usepackage{ifthen}
\usepackage{graphicx}
\usepackage{amsmath}
\usepackage[T1]{fontenc} 
\usepackage[utf8]{inputenc}
\usepackage[usenames,dvipsnames]{color}
\usepackage{color}
\usepackage[english]{babel}
\usepackage{fancyhdr}
\usepackage{fancybox}
\usepackage{tikz}

\theoremstyle{plain}
\newtheorem{prop}{Proposition}

\newtheorem{ques}{Question}

\newtheorem{conj}{Conjecture}

\theoremstyle{definition}
\newtheorem{defi}{Definition}[section]

\newtheorem{cor}{Corollary}[section]
\newtheorem{thm}{Theorem}[section]

\newtheorem{lem}{Lemma}[section]
\newtheorem{prob}{Problem}
\newtheorem{rem}{Remark}[section]

\theoremstyle{plain}

\newtheorem*{lemA}{Lemma A}

\newcounter{minutes}
\divide\time by 60
\newcounter{hours}
\multiply\time by 60
\addtocounter{minutes}{-\time}

\newcounter {own}
\def\theown {\thesection       .\arabic{own}}

\newenvironment{pf}[1][]{%
	\vskip 3mm
	\noindent
	\ifthenelse{\equal{#1}{}}%
	{{\slshape Proof. }}%
	{{\slshape #1.} }%
}%
{\qed\bigskip}

\newcounter{alphabet}

\newcommand{\Z}{{\mathbb Z}}

\def\be{\begin{equation}}
	\def\ee{\end{equation}}

\newcommand{\bee}{\begin{enumerate}}
	\newcommand{\eee}{\end{enumerate}}

\newcommand{\blem}{\begin{lem}}
	\newcommand{\elem}{\end{lem}}
\newcommand{\bthm}{\begin{thm}}
	\newcommand{\ethm}{\end{thm}}
\newcommand{\bcor}{\begin{cor}}
	\newcommand{\ecor}{\end{cor}}
\newcommand{\beg}{\begin{examp}}
	\newcommand{\eeg}{\end{examp}}
\newcommand{\begs}{\begin{examples}}
	\newcommand{\eegs}{\end{examples}}

\newcommand{\bdefn}{\begin{defn}}
	\newcommand{\edefn}{\end{defn}}

\newcommand{\bprob}{\begin{prob}}
	\newcommand{\eprob}{\end{prob}}
\newcommand{\bei}{\begin{itemize}}
	\newcommand{\eei}{\end{itemize}}

\newcommand{\bcon}{\begin{conj}}
	\newcommand{\econ}{\end{conj}}
\newcommand{\bcons}{\begin{conjs}}
	\newcommand{\econs}{\end{conjs}}
\newcommand{\bprop}{\begin{prop}}
	\newcommand{\eprop}{\end{prop}}
\newcommand{\br}{\begin{rem}}
	\newcommand{\er}{\end{rem}}
\newcommand{\brs}{\begin{rems}}
	\newcommand{\ers}{\end{rems}}
\newcommand{\bo}{\begin{obser}}
	\newcommand{\eo}{\end{obser}}
\newcommand{\bos}{\begin{obsers}}
	\newcommand{\eos}{\end{obsers}}
\newcommand{\bpf}{\begin{pf}}
	\newcommand{\epf}{\end{pf}}
\newcommand{\ba}{\begin{array}}
	\newcommand{\ea}{\end{array}}
\newcommand{\beq}{\begin{eqnarray}}
	\newcommand{\beqq}{\begin{eqnarray*}}
		\newcommand{\eeq}{\end{eqnarray}}
	\newcommand{\eeqq}{\end{eqnarray*}}

\makeatletter
\@namedef{subjclassname@2020}{%
  \textup{2020} Mathematics Subject Classification}
\makeatother

\begin{document}

\title[Refined Bohr inequalities]{Refined Bohr inequalities and a refined Bohr-Rogosinski inequality 
on complex Banach spaces}

\author{Molla Basir Ahamed}
\address{Molla Basir Ahamed, Department of Mathematics, Jadavpur University, Kolkata-700032, West Bengal, India.}
\email{mbahamed.math@jadavpuruniversity.in}

\author{Sabir Ahammed}
\address{Sabir Ahammed, Department of Mathematics, Jadavpur University, Kolkata-700032, West Bengal, India.}
\email{sabira.math.rs@jadavpuruniversity.in}

\author{Hidetaka Hamada$^{\ast}$}
\thanks{$^{\ast}$ Corresponding author}
\address{Hidetaka Hamada, Faculty of Science and Engineering, Kyushu Sangyo University, 3-1 Matsukadai 2-Chome Higashi-Ku, Fukuoka 813-8503, Japan.}
\email{h.hamada@ip.kyusan-u.ac.jp}


\def\thefootnote{}
\footnotetext{ {\tiny File:~\jobname.tex,
printed: \number\year-\number\month-\number\day,
          \thehours.\ifnum\theminutes<10{0}\fi\theminutes }
} \makeatletter\def\thefootnote{\@arabic\c@footnote}\makeatother

\begin{abstract} 
In this paper, we first establish refined versions of the Bohr inequalities for the class of holomorphic functions from the unit ball $B_X$ of a complex Banach space $X$ into $\mathbb{C}$. 
As applications, 
we will establish refined Bohr inequalities of functional type or of norm type
for holomorphic mappings with lacunary series on the unit ball  $B_X$ with values in higher dimensional spaces.
Next, we obtain the Bohr-Rogosinski inequality for the class  of holomorphic functions on $B_X.$ 
In addition, we establish an improved version of the Bohr inequality for holomorphic functions on $B_X$.  
All the results are proved to be sharp.
\end{abstract}

\maketitle
\pagestyle{myheadings}
\markboth{  M. B. Ahamed, S. Ahammed, and H. Hamada }{Refined Bohr inequalities}

\section{Introduction}
\label{intro}
Throughout the paper, we denote the set of non-negative integers
by $\mathbb{N}_0$
and the set of positive integers by $\mathbb{N}$.
For a real number $r$,
let $\lfloor{r}\rfloor$ denote the floor function,
i.e., 
the greatest integer less than or equal to $r$.
Let $X$ and $Y$ be  complex Banach spaces with norms $||\cdot||_X$
and $||\cdot||_Y$, respectively.
For simplicity, we omit the subscript for the norm
when it is obvious from the context. 
Let $B_X$ and $B_Y$ be the open unit balls in $X$ and $Y$, respectively. 
If $X=\mathbb{C},$ then $B_X=\mathbb{D}=\{z\in\mathbb{C}:|z|<1\}$
is the unit disk in $\mathbb{C}$.
For a domain $D\subset X$ and a set $E\subset Y$, 
we denote the set of holomorphic mappings from $D$ into $E$
by $\mathcal{H}\left(D,E\right)$.
$\overline{D}$ denotes the closure of $D$.
For each $x\in X\setminus\{0\},$ we define
\begin{align*}
	T(x)=\{T_x\in X^*:||T_x||=1, T_x(x)=||x||\}, 
\end{align*}
where
$X^*$ is the dual space of $X$.
Then the well known Hahn-Banach theorem implies that $T(x)$ is non empty.
\begin{defi}
	Let $X$ and $Y$ be complex Banach spaces. Let $k\in \mathbb{N}$. A mapping $P:X\rightarrow Y$ is called a homogeneous polynomial of degree $k$ if there exists a $k$-linear mapping $u$ from $X^k$ into $Y$ such that 
	\begin{align*}
		P(x)=u(x,\dots,x)
	\end{align*}
	for every $x\in X.$ 
\end{defi}

Throughout our discussion, the degree of a homogeneous polynomial is denoted by a subscript. We note that if $P_m$ is an $m$-homogeneous polynomial from $X$ into $Y,$ there uniquely exists a symmetric $m$-linear mapping $u$ with 
\begin{align*}
	P_m(x)=u(x,\dots,x).
\end{align*}
A holomorphic function $f : {B}_X\to\mathbb{D}$ can be expressed in a homogeneous polynomial expansion $f(z)=\sum_{s=0}^{\infty}P_{s}(z)$, $z\in {B}_X$.

For  $F\in \mathcal{H}\left(B_X,Y\right)$ and $z\in B_X$, let $D^kF(z)$
denote the $k$-th Fr\'{e}chet derivative of $F$ at $z$.
Then any holomorphic mapping  $F\in \mathcal{H}\left(B_X,Y\right)$
can be expanded into the series
\begin{equation}
\label{expansion}
F(z) =\sum_{k=0}^{\infty}\frac{1}{k!}D^k F(0)(z^k)
\end{equation}
in a neighbourhood of the origin. 
Note that if $F(B_X)$ is bounded,
then (\ref{expansion}) converges uniformly on $rB_X$ for each $r\in (0,1)$.

\vspace{1.2mm} 

Let us start with a remarkable result of Harald Bohr published in $ 1914 $, dealing with a problem connected with Dirichlet series and number theory, which stimulated a lot of research activity into geometric function theory.
\begin{theoA}\cite{Bohr-1914}
	If $ f(z)=\sum_{s=0}^{\infty}a_sz^s\in\mathcal{H}\left(\mathbb{D},\mathbb{D}\right) $, then 
	\begin{equation}\label{e-1.2}
		\sum_{s=0}^{\infty}|a_s|r^s\leq 1 \;\; \mbox{for}\;\; |z|=r\leq {1}/{3}.
	\end{equation}
\end{theoA}
The inequality fails when $ r>{1}/{3} $ in the sense that there are functions in $ \mathcal{H}\left(\mathbb{D},\mathbb{D}\right) $ for which the inequality is reversed when $ r>{1}/{3} $. H. Bohr initially showed the inequality \eqref{e-1.2} holds only for $|z|\leq1/6$, which was later improved independently by M. Riesz, I. Schur, F. Wiener and some others. The sharp constant $1/3$  and the inequality \eqref{e-1.2} in Theorem A are called respectively, the Bohr radius and the  classical Bohr inequality for the family $\mathcal{H}\left(\mathbb{D},\mathbb{D}\right)$.  Several other proofs of this interesting inequality were given in different articles  (see \cite{Sidon-1927,Paulsen-Popescu-Singh-PLMS-2002,Tomic-1962}).\vspace{1.2mm}

Similar to the Bohr radius, the notion of the Rogosinski radius was first introduced in \cite{Rogosinski-1923} for functions
 $f\in\mathcal{H}\left(\mathbb{D},\mathbb{D}\right)$. 
 Nevertheless,  as compared to the Bohr radius, the Rogosinski radius has not received the same level of research attention. 
If $ B $ and $ R $
denote the Bohr radius and the Rogosinski radius, respectively, then it is easy to see that $ B = 1/3 < 1/2 = R $.
 
Moreover, analogous to the Bohr inequality, there is also a concept of the Bohr-Rogosinski inequality. Following the article \cite{Kayumov-Khammatova-Ponnusamy-JMAA-2021},
for the functions $ f(z)=\sum_{s=0}^{\infty}a_sz^s\in\mathcal{H}\left(\mathbb{D},\mathbb{D}\right) $, the Bohr-Rogosinski sum $ R^f_N(z) $ of $ f $ is defined by 
\begin{align}\label{e-0.2}
	R^f_N(z):=|f(z)|+\sum_{s=N}^{\infty}|a_s|r^s, \; |z|=r.
\end{align}
An interesting fact to be observed is that for $ N=1 $, the quantity in \eqref{e-0.2} is related to the classical Bohr sum in which $ |f(0)| $ is replaced by $ |f(z)| $. The relation $ R^f_N(z)\leq 1 $ is called the Bohr-Rogosinski inequality. 
For some recent development on the Bohr-Rogosinski inequality, the reader is refereed to the article  
\cite{ Allu-Arora-JMAA-2022}, and the references therein.\vspace{1.2mm}

In recent years, refining the Bohr-type inequalities
have been an active research topic. Many researchers continuously investigated refined Bohr-type inequalities and also examining their sharpness for certain classes of analytic functions, for classes of harmonic mappings on the unit disk $ \mathbb{D} $, or on the shifted disk $ \Omega_{\gamma}:=\{z\in\mathbb{C} : |z+\frac{\gamma}{1-\gamma}|<\frac{1}{1-\gamma}\} $ where $ 0\leq\gamma<1 $, for operator-valued functions. 
For detailed information on such studies, the readers are referred to  \cite{Ahamed-AASFM-2022,Liu-Ponnusamy-Wang-RACSAM-2020,Liu-Liu-Ponnusamy-BSM-2021,Ponnusamy-Vijayakumar-Wirths-JMAA-2022} and the references therein. 
In particular, 
Liu \emph{et al.} \cite{Liu-Liu-Ponnusamy-BSM-2021}
investigated refined Bohr-Rogosinski type inequalities for holomorphic functions and refined Bohr inequalities for holomorphic functions with 
lacunary series on the unit disk.

Let $f$ be holomorphic in $\mathbb D$, and for $0<r<1$,  let $\mathbb D_r=\{z\in \mathbb C: |z|<r\}$.
Let $S_r:=S_r(f)$ denote the planar integral
\begin{align*}
	S_r:=\int_{\mathbb D_r} |f'(z)|^2 d A(z).
\end{align*}
If the function $f\in \mathcal{H}\left(\mathbb{D},\mathbb{D}\right)$ has Taylor's series expansion $f(z)=\sum_{s=0}^{\infty}a_sz^s $, then we obtain  (see \cite{Kayumov-Ponnusamy-CRACAD-2018})
\begin{equation}
\label{Sr}
	S_r= \pi\sum_{s=1}^\infty s|a_s|^2 r^{2s}.
\end{equation}
In the study of the improved Bohr inequality, the quantity $ S_r $ plays a significant role. There are many results on the improved Bohr inequality for the class $ \mathcal{H}\left(\mathbb{D},\mathbb{D}\right) $ (see \cite{Ismagilov- Kayumov- Ponnusamy-2020-JMAA,Kayumov-Ponnusamy-CRACAD-2018,Liu-Liu-Ponnusamy-BSM-2021}), and for harmonic mappings on unit disk (see \cite{Evdoridis-Ponnusamy-Rasila-Indag.Math.-2019}).
 
The Bohr phenomenon has been extended to
holomorphic or pluriharmonic functions of several variables (see e.g. 
\cite{Arora-CVEE-2022,Boas-Khavinson-PAMS-1997,Liu-Ponnusamy-PAMS-2021,Hamada-Math.Nachr.-2021,Hamada- Honda-BMMS-2024,Aizn-PAMS-2000,Hamada-Honda-Mizota-MIA-2020,Hamada-IJM-2009,Liu-Liu-JMAA-2020,Lin-Liu-Ponnusamy-Acta-2023,Kumar-PAMS-2022,Hamada-Honda-2024}).
In recent times, the study of the Bohr inequality in complex Banach spaces is an active research area, and many researchers have investigated it. 

In order to determine the Bohr radius for the class of odd functions in the family $ \mathcal{H}\left(\mathbb{D},\mathbb{D}\right) $, which was posed in \cite{Ali-2017}, Kayumov and Ponnusamy 
\cite{Kayumov-Ponnusamy-CMFT-2017,Kayumov-Ponnusamy-2018-JMAA}
studied  the Bohr inequalities for holomorphic functions with lacunary series in a single complex variable.
Generalizations of this result to holomorphic mappings in several complex variables have been studied (see e.g. 
\cite{Arora-CVEE-2022,Hamada-Honda-Mizota-MIA-2020,Lin-Liu-Ponnusamy-Acta-2023,Liu-Liu-JMAA-2020}).
The third author \emph{et al.} in \cite{Hamada-Honda-Mizota-MIA-2020} have established the Bohr inequality for a class of holomorphic functions $f : B_X\to\mathbb{D}$ with the homogeneous polynomial expansion $f(z)=\sum_{s=0}^{\infty}P_{sp+m}(z)$, $z\in {B}_X$ and obtained the following inequality.

\begin{theoB}(see \cite[Theorem 4.2]{Hamada-Honda-Mizota-MIA-2020})
	Let $p\in\mathbb{N}$, $m\in\mathbb{N}_0$ with $0\leq m\leq p$ and $f : {B}_X\to\mathbb{D}$ be a holomorphic function with the homogeneous polynomial expansion $f(z)=\sum_{s=0}^{\infty}P_{sp+m}(z)$, $z\in {B}_X$. Then the following sharp inequality holds:
	\begin{align*}
		\sum_{s=0}^{\infty}|P_{sp+m}(z)|\leq 1
	\end{align*}
	for $||z||=r\leq r_{p,m}$, where $r_{p,m}$ is the maximal root of the equation
	\begin{align*}
		-6r^{p-m}+r^{2(p-m)}+8r^{2p}+1=0
	\end{align*}
	in $(0,1).$ The number $r_{p,m}$ is sharp. 
\end{theoB}

 In \cite{Liu-Liu-JMAA-2020}, Liu and Liu investigated the Bohr inequality  of functional type for holomorphic mappings $f \in \mathcal{H}\left(B_X,\overline{B_X}\right)$ with lacunary series and obtained the following result.

\begin{theoC}\cite[Theorem 3.4]{Liu-Liu-JMAA-2020}\label{BS-thm-2.4}
	Let $m\in \mathbb{N}_0$, $p\in\mathbb{N}$ with $0\leq m\leq p$, and let
	\begin{align*}
		f(z)=\frac{D^mf(0)\left(z^m\right)}{m!}+\sum_{s=1}^{\infty}\frac{D^{sp+m}f(0)\left(z^{sp+m}\right)}{(sp+m)!}\in\mathcal{H}\left({B}_X,\overline{{B}_X}\right).
	\end{align*}
	Then 
	\begin{align*}
		\frac{|T_v(D^mf(0)\left(z^m\right))|}{m!}+\sum_{s=1}^{\infty}\frac{|T_v(D^{sp+m}f(0)\left(z^{sp+m}\right))|}{(sp+m)!}\leq 1
	\end{align*}
	for $||z||=r\leq r_{p,m}$, where $ r_{p,m} $ is the same as Theorem B, $v\in X$ is fixed with $||v||=1.$
	Each $r_{p,m}$ is sharp.
\end{theoC}

The Bohr inequality of norm type for holomorphic mappings with lacunary series under certain restricted conditions on the mapping $f$ on the unit ball of a complex Banach space $X$ was investigated, and the following result was obtained.

\begin{theoD}(\cite[Theorem 3.3]{Liu-Liu-JMAA-2020})
	Let $p\in\mathbb{N},$  $m\in \mathbb{N}_{0}$ with $0\leq m\leq p$, 
	\begin{align*}
		f(z)=zg(z)=\frac{D^mf(0)\left(z^m\right)}{m!}+\sum_{s=1}^{\infty}\frac{D^{sp+m}f(0)\left(z^{sp+m}\right)}{(sp+m)!}\in\mathcal{H}\left(B_X,\overline{B_X}\right),
	\end{align*}
	where $g\in\mathcal{H}(B_X, \mathbb{C})$.	Then 
	\begin{align*}
		\frac{||D^mf(0)\left(z^m\right)||}{m!}+\sum_{s=1}^{\infty}\frac{||D^{sp+m}f(0)\left(z^{sp+m}\right)||}{(sp+m)!}\leq 1
	\end{align*}
	for $||z||=r\leq r_{p,m}$, where $ r_{p,m} $ is the same as Theorem B.	Each $r_{p,m}$ is sharp.
\end{theoD}

\begin{remark}
Note that Theorem D is sharp only for $m\in \mathbb{N}$, 
because $f(z)=zg(z)$ implies that $f(0)=0$.
\end{remark}
\vspace{2mm}

 In \cite{Liu-Liu-JMAA-2020}, Liu and Liu investigated the Bohr inequality of functional type for holomorphic mappings $f \in \mathcal{H}\left(B_X,\overline{B_X}\right)$
 and obtained the following result.
 
\begin{theoE}\cite[Theorem 3.2]{Liu-Liu-JMAA-2020}\label{BS-thm-2.2}
	Let $m\in\mathbb{N}_0$, $N\geq m+1$ and 
	\begin{align*}
		f(z)=\frac{D^mf(0)\left(z^m\right)}{m!}+\sum_{s=N}^{\infty}\frac{D^sf(0)\left(z^s\right)}{s!}\in\mathcal{H}\left({B}_X, \overline{{B}_X}\right).
	\end{align*}
	Then 
	\begin{align*}
		\frac{|T_v\left(D^mf(0)\left(z^m\right)\right)|}{m!}+\sum_{s=N}^{\infty}\frac{|T_v(D^sf(0)\left(z^s\right))|}{s!}\leq 1
	\end{align*}
	for $||z||=r\leq r_{N,m}^*$ and, $v\in X$ is fixed with $||v||=1,$ where $ r_{N,m}^* $ is the maximal positive root in $(0,1)$ of the equation
	\begin{align*}
		\begin{cases}
			2r^N+r-1=0, m=0,\\
			4r^{2(N-m)}+4r^{N+1-2m}-4r^{N-2m}+r^2-2r+1=0,  N>2m>1,\\
			4r^N+r^{2+2m-N}-2r^{1+2m-N}+r^{2m-N}+4r-4=0,  m+1\leq N\leq 2m.	
		\end{cases}
	\end{align*}
	Especially, $r_{1,0}^*=1/3$ and $r_{2,1}^*=3/5$. 
\end{theoE}

\begin{remark}
For  $m, N\in \Z$ with $m\geq 0$ and $N\geq m+1$,
let ${r}_{N,m}^{**}$ be the maximal positive root of the equation
\begin{equation}
\label{r*}
4r^{2N-m}+4r^{N+1-m}-4r^{N-m}+r^{m+2}-2r^{m+1}+r^m=0.
\end{equation}
It is simple 
 to check that $r_{N,m}^{**}=r_{N,m}^*$ 
and ${r}_{N,m}^{**}\in (0,1)$
(see e.g. \cite{Hamada-Honda-2024}).
\end{remark}

 In \cite{Liu-Liu-JMAA-2020}, Liu and Liu investigated the Bohr inequality of norm type for holomorphic mappings $f \in \mathcal{H}\left(B_X,\overline{B_X}\right)$ under certain restricted conditions on the mapping $f$ and obtained the following result.
 
\begin{theoF}(\cite[Theorem 3.1]{Liu-Liu-JMAA-2020})
	Let $m\in\mathbb{N}_0$, $N\geq m+1$ and 
	\begin{align*}
		f(z)=zg(z)=\frac{D^mf(0)\left(z^m\right)}{m!}+\sum_{s=N}^{\infty}\frac{D^sf(0)\left(z^s\right)}{s!}\in\mathcal{H}\left(B_X, \overline{B_X}\right),
	\end{align*}
	where $g\in\mathcal{H}(B_X, \mathbb{C})$. Then 
	\begin{align*}
		\frac{||D^mf(0)\left(z^m\right)||}{m!}+\sum_{s=N}^{\infty}\frac{||D^sf(0)\left(z^s\right)||}{s!}\leq 1
	\end{align*}
	for $||z||=r\leq r_{N,m}^*$,   where $ r_{N,m}^* $ is the same as Theorem $E$.
\end{theoF}

\begin{remark}
Also, note that Theorem F may not be sharp for $m=0$, 
because $f(z)=zg(z)$ implies that $f(0)=0$.
\end{remark}
\vspace{2mm}

\section{Main results}

No one has yet explored what could be a refined version of the Bohr inequality for holomorphic functions $f : {B}_X\to\mathbb{D}$ with the homogeneous polynomial expansion $f(z)=\sum_{s=0}^{\infty}P_{sk+m}(z)$, $z\in {B}_X$. Inspired by the mentioned articles
in Section \ref{intro}, we are interested in investigating refined versions of the Bohr inequalities of Theorem B, C, D and also in establishing its sharpness. Hence, the following question arises naturally.

\begin{ques}\label{q-1.1}
	Can we establish refined versions of Theorems B,  C and D?
	 Can we show them sharp?
\end{ques}

To give an answer to Question \ref{q-1.1}, in the next result, we first establish a refined version of the Bohr inequality for holomorphic functions $f : {B}_X\to\overline{\mathbb{D}}$ with the homogeneous polynomial expansion $f(z)=\sum_{s=0}^{\infty}P_{sk+m}(z)$, $z\in {B}_X$
(cf. \cite[Theorem 3]{Liu-Liu-Ponnusamy-BSM-2021} in the case $B_X=\mathbb{D}$).

\begin{thm}\label{Th-3.22}
	Let $p\in\mathbb{N}$, $m\in\mathbb{N}_0$ with $0\leq m\leq p,$ and $f : {B}_X\to\overline{\mathbb{D}}$ be a holomorphic function with the homogeneous polynomial expansion $f(z)=\sum_{s=0}^{\infty}P_{sp+m}(z)$, $z\in {B}_X$. Then the following sharp inequality holds:
	\begin{align*}
		\mathcal{A}^f_{p, m}(z):=\sum_{s=0}^{\infty}|P_{sp+m}(z)|+\left(\frac{1}{r^m+|P_m(z)|}+\frac{r^{p-m}}{1-r^p}\right)\sum_{s=1}^{\infty}|P_{sp+m}(z)|^2\leq 1
	\end{align*}
	for $||z||=r\leq r_{p,m}^{***}$ and $r_{p,m}^{***}$ is the maximal positive root in $(0,1)$ of the equation
	\begin{equation}\label{eq-1-2-b}
		5r^{2p+m}-2r^{p+m}+r^{m}+4r^{2p}-4r^{p}=0.
	\end{equation}
	The constant $ r_{p,m}^{***} $ is best possible.
\end{thm}

As a corollary of Theorem \ref{Th-3.22}, we shall establish a refined version of the  Bohr inequality of functional type for holomorphic mappings with lacunary series on the unit ball ${B}_X$. 
We obtain the following result as a sharp refinement of Theorem C.

\begin{thm}\label{BS-thm-2.8}
	Let $m\in \mathbb{N}_0$, $p\in\mathbb{N}$ with $0\leq m\leq p$ and let
	\begin{align*}
		f(z)=\frac{D^mf(0)\left(z^m\right)}{m!}+\sum_{s=1}^{\infty}\dfrac{D^{sp+m}f(0)\left(z^{sp+m}\right)}{(sp+m)!}\in\mathcal{H}\left({B}_X,\overline{{B}_Y}\right).
	\end{align*}
	Then $\mathcal{B}^{p,m}_f(z)\leq 1$ for $||z||=r\leq r_{p,m}^{***}$, where $r_{p,m}^{***}$ is the maximal positive root of the equation \eqref{eq-1-2-b}
	in $(0,1)$ and
	\begin{align*}
		\mathcal{B}^{p,m}_f(z):&=\frac{|T_v\left(D^mf(0)\left(z^m\right)\right)|}{m!}+\sum_{s=1}^{\infty}\frac{|T_v(D^{sp+m}f(0)\left(z^{sp+m}\right))|}{(sp+m)!}\\& \quad+\left(\frac{1}{||z||^m+\frac{|T_v\left(D^mf(0)\left(z^m\right)\right)|}{m!}}+\frac{||z||^{p-m}}{1-||z||^p}\right)\sum_{s=1}^{\infty}\left(\frac{|T_v\left(D^{sp+m}f(0)\left(z^{sp+m}\right)\right)|}{(sp+m)!}\right)^2,
	\end{align*}
	$v\in Y$ is fixed with $||v||=1$.
	 Each $r_{p,m}^{***}$ is sharp.
\end{thm}

As a corollary of Theorem \ref{Th-3.22}, we obtain a sharp refined version of the Bohr inequality of norm type for holomorphic mappings $f\in\mathcal{H}\left(B_X,\overline{B_X}\right)$
with lacunary series under certain restricted conditions on the mappings $f$,
which gives a sharp refinement of Theorem D.

\begin{thm}\label{BS-thm-2.7}
	Let  $p\in \mathbb{N},$ $m\in\mathbb{N}_0$ with $0\leq m\leq p$ and
	\begin{align*}
		f(z)=zg(z)=\frac{D^mf(0)\left(z^m\right)}{m!}+\sum_{s=1}^{\infty}\frac{D^{sp+m}f(0)\left(z^{sp+m}\right)}{(sp+m)!}\in\mathcal{H}\left(B_X,\overline{B_X}\right),
	\end{align*}
	where $g\in\mathcal{H}(B_X, \mathbb{C})$.	Then 
	\begin{align*}
		\mathcal{C}^{f}_{p, m}(z):&=\frac{||D^mf(0)\left(z^m\right)||}{m!}+\sum_{s=1}^{\infty}\frac{||D^{sp+m}f(0)\left(z^{sp+m}\right)||}{(sp+m)!}\\& \quad+\left(\frac{1}{||z||^m+\frac{||D^mf(0)(z^m)||}{m!}}+\frac{||z||^{p-m}}{1-||z||^p}\right)\sum_{s=1}^{\infty}\left(\frac{||D^{sp+m}f(0)(z^{sp+m})||}{(sp+m)!}\right)^2
		\\
		&\leq 1
	\end{align*}
	for $||z||=r\leq r_{p,m}^{***}$, where $r_{p,m}^{***}$ is the maximal positive root of the equation \eqref{eq-1-2-b} in $(0,1)$. 	Each $r_{p,m}^{***}$ is sharp for $p,m\in \mathbb{N}$ with $1\leq m\leq p$.
\end{thm}

Now, we are interested in investigating refined versions of Theorems E and F for holomorphic mappings on the unit ball $B_X$ of complex Banach spaces $X$, and also in establishing its sharpness using techniques found in other works, such as \cite{Hamada-Honda-2024}. Hence, the following question arises naturally.
\begin{ques}\label{BS-qn-2.1}
	Can we establish refined versions of Theorems E and F? Can we show them sharp keeping the radius unchanged? 
\end{ques}

To give an answer to Question \ref{BS-qn-2.1}, we first establish a new refined version of the Bohr inequality for 
$f(z)=P_m(z)+\sum_{s=N}^{\infty}P_s(z)\in\mathcal{H}\left({B}_X, \overline{\mathbb{D}}\right)$. 

\begin{thm}\label{BS-thm-2.6a}
	Let $m\in\mathbb{N}_0$, $N\geq m+1$ and
	\begin{align*}
		f(z)=P_m(z)+\sum_{s=N}^{\infty}P_s(z)\in\mathcal{H}\left({B}_X, \overline{\mathbb{D}}\right).
	\end{align*}
	Then 
	\begin{align*}
		\mathcal{D}_{N,m}^f(z):&=|P_m(z)|+\sum_{s=N}^{\infty}|P_s(z)|\\& \quad+\left(\frac{1}{||z||^m+|P_m(z)|}+\frac{||z||^{1-m}}{1-||z||}\right)\sum_{s=N}^{\infty}|P_s(z)|^2\leq 1
	\end{align*}
	for $||z||=r\leq r_{N,m}^{**}$, where ${r}_{N,m}^{**}$ is the maximal positive root of the equation
\eqref{r*}.
Each $r_{m+1,m}^{**}$ is sharp.	
\end{thm}

As a corollary of Theorem \ref{BS-thm-2.6a}, we establish a new refined version of the Bohr inequality of functional type for holomorphic mappings on the unit ball $B_X$. 
We obtain the following result as a sharp refinement of Theorem E.
\vspace{2mm}

\begin{thm}\label{BS-thm-2.6}
	Let $m\in\mathbb{N}_0$, $N\geq m+1$ and
	\begin{align*}
		f(z)=\frac{D^mf(0)\left(z^m\right)}{m!}+\sum_{s=N}^{\infty}\frac{D^sf(0)\left(z^s\right)}{s!}\in\mathcal{H}\left({B}_X, \overline{{B}_Y}\right).
	\end{align*}
	Then 
	\begin{align*}
		\mathcal{E}_{N,m}^f(z):&=\frac{|T_v\left(D^mf(0)\left(z^m\right)\right)|}{m!}+\sum_{s=N}^{\infty}\frac{|T_v\left(D^sf(0)\left(z^s\right)\right)|}{s!}\\& \quad+\left(\frac{1}{||z||^m+\frac{|T_v\left(D^mf(0)(z^m)\right)|}{m!}}+\frac{||z||^{1-m}}{1-||z||}\right)\sum_{s=N}^{\infty}\left(\frac{|T_v\left(D^sf(0)(z^s)\right)|}{s!}\right)^2\leq 1
	\end{align*}
	for $||z||=r\leq r_{N,m}^{**}$, where ${r}_{N,m}^{**}$ is the maximal positive root of the equation
\eqref{r*} and $v\in Y$ is fixed with $||v||=1$.
Each $r_{m+1,m}^{**}$ is sharp.	
\end{thm}

As a corollary of Theorem \ref{BS-thm-2.6a}, we establish a new refined version of the Bohr inequality of norm type for holomorphic mappings on the unit ball $B_X$ under the restricted conditions that have been considered in \cite{Liu-Liu-JMAA-2020}. 
We obtain the following result as a sharp refinement of Theorem F.
\vspace{2mm}

\begin{thm}\label{BS-thm-2.5}
	Let $m\in\mathbb{N}_0$, $N\geq m+1$ and
	\begin{align*}
		f(z)=zg(z)=\frac{D^mf(0)\left(z^m\right)}{m!}+\sum_{s=N}^{\infty}\frac{D^sf(0)\left(z^s\right)}{s!}\in\mathcal{H}\left(B_X, \overline{B_X}\right),
	\end{align*}
	where $g\in\mathcal{H}(B_X, \mathbb{C})$. Then 
	\begin{align*}
		\mathcal{F}^{f}_{N, m}(z):&=\frac{||D^mf(0)\left(z^m\right)||}{m!}+\sum_{s=N}^{\infty}\frac{||D^sf(0)\left(z^s\right)||}{s!}
		\\& \quad
		+\left(\frac{1}{||z||^m+\frac{||D^mf(0)(z^m)||}{m!}}+\frac{||z||^{1-m}}{1-||z||}\right)\sum_{s=N}^{\infty}\left(\frac{||D^{s+m}f(0)(z^{s+m})||}{(s+m)!}\right)^2\leq 1
	\end{align*}
	for $||z||=r\leq r_{N,m}^{**}$, where ${r}_{N,m}^{**}$ is the maximal positive root of the equation $\eqref{r*} $.
	Each $r_{m+1,m}^{**}$ is sharp for $m\in \mathbb{N}$.
\end{thm}

Note that for $m=0,$ $N=1,$ we obtain the following corollaries of the above theorems.

\begin{cor}
	Let 
	$	f(z)=\sum_{s=0}^{\infty}P_s(z)\in\mathcal{H}\left({B}_X, \overline{\mathbb{D}}\right).$ Then
	\begin{align*}
		&|f(0)|+\sum_{s=1}^{\infty}|P_s(z)|
		+\left(\frac{1}{1+|f(0)|}+\frac{||z||}{1-||z||}\right)
		\sum_{s=1}^{\infty}|P_s(z)|^2\leq 1
	\end{align*}
	for $||z||=r\leq 1/3$. The constant $1/3$ is optimal.
\end{cor}

\begin{cor}
	Let 
	$	f(z)=\sum_{s=0}^{\infty}\dfrac{D^sf(0)\left(z^s\right)}{s!}\in\mathcal{H}\left({B}_X, \overline{{B}_Y}\right).$ Then
	\begin{align*}
		&|T_v(f(0))|+\sum_{s=1}^{\infty}\frac{|T_v\left(D^sf(0)\left(z^s\right)\right)|}{s!}\\&\quad+\left(\frac{1}{1+|T_v(f(0))|}+\frac{||z||}{1-||z||}\right)\sum_{s=1}^{\infty}\left(\frac{|T_v\left(D^sf(0)(z^s)\right)|}{s!}\right)^2\leq 1
	\end{align*}
	for $||z||=r\leq 1/3$ and $v\in Y$ is fixed with $||v||=1$. The constant $1/3$ is optimal.
\end{cor}


For the study of the Bohr-Rogosinski inequality on complex Banach spaces, we are interested in giving an answer to the following question.
\begin{ques}\label{q-1.2}
	Can we obtain a refined Bohr-Rogosinski inequality for holomorphic functions $f : {B}_X\to\mathbb{D}$ with the homogeneous polynomial expansion $f(z)=\sum_{s=0}^{\infty}P_{s}(z)$, $z\in {B}_X$? 
\end{ques}

Let $F:B_X\to Y$ be a holomorphic mapping.
For $k\in {\mathbb N}$, 
we say that $z=0$ is a {\it zero of order $k$ of $F$} if
$F(0)=0$, $DF(0)=0$, $\dots$, $D^{k-1}F(0)=0$,
but $D^kF(0)\neq 0$.

A holomorphic mapping $v:B_X\to B_Y$ with $v(0)=0$
is called a {\it Schwarz mapping}.
We note that if $v$ is a Schwarz mapping such that $z=0$ is a zero of order $k$ of $v$,
then the following estimation holds (see e.g. \cite[Lemma 6.1.28]{Graham-2003}):
\begin{equation}
\label{Schwarz-k}
\| v(z)\|_Y\leq \| z\|_X^k,
\quad z\in B_X.
\end{equation}

Our next result is the following which gives an answer to Question \ref{q-1.2} completely
(cf. 
\cite[Theorem 1]{Chen-Liu-Ponnusamy-RM-2023},
\cite[Theorem 1]{Liu-Liu-Ponnusamy-BSM-2021}
in the case $B_X=\mathbb{D}$).

\begin{thm}\label{Th-3.2}
	Let $f : {B}_X\to\mathbb{D}$ be a holomorphic function with the homogeneous polynomial expansion $f(z)=\sum_{s=0}^{\infty}P_{s}(z)$, $z\in {B}_X$. Then for $m, N\in\mathbb{N}$ and $p\in (0, 2]$, we have the following sharp inequality:
	\begin{align*}
		\mathcal{G}^{f}_{m, p, N}(z):&=|f(v(z))|^p+\sum_{s=N}^{\infty}|P_{s}(z)|+\mathrm{sgn}(t)\sum_{s=1}^{t}|P_s(z)|^2\frac{r^{N-2s}}{1-r}\\&+\left(\frac{1}{1+|f(0)|}+\frac{r}{1-r}\right)\sum_{s=t+1}^{\infty}|P_{s}(z)|^2\leq 1
	\end{align*}
	for $||z||=r\leq R^N_{p,m}$, where $v:B_X\to B_X$ is a  Schwarz mapping such that $z=0$ is a zero of order $m$ of $v$, $t=\lfloor{(N-1)/2}\rfloor$
	and 
	$R^N_{p,m}$ is the unique positive root in $(0,1)$ of the equation
	\begin{align}\label{BS-ee-3.1}
		p\left(\frac{1-r^m}{1+r^m}\right)-\frac{2r^N}{1-r}=0.
	\end{align}
	The constant $R^N_{p,m}$ is best possible.
\end{thm}

Letting $m\to \infty$, we obtain the following result.

\begin{cor}\label{Th-3.3}
	Let $f : {B}_X\to\mathbb{D}$ be a holomorphic function with the homogeneous polynomial expansion $f(z)=\sum_{s=0}^{\infty}P_{s}(z)$, $z\in {B}_X$. Then for any $N\in\mathbb{N}$ and $p\in (0, 2]$, we have the following sharp inequality:
	\begin{align*}
		\mathcal{H}_{p,N}^f(z):&=|f(0)|^p+\sum_{s=N}^{\infty}|P_{s}(z)|+\mathrm{sgn}(t)\sum_{s=1}^{t}|P_s(z)|^2\frac{r^{N-2s}}{1-r}\\&+\left(\frac{1}{1+|f(0)|}+\frac{r}{1-r}\right)\sum_{s=t+1}^{\infty}|P_{s}(z)|^2\leq 1
	\end{align*}
	for $||z||=r\leq R^N_{p}$, where $t=\lfloor{(N-1)/2}\rfloor$ and  $R^N_{p}$ is the unique positive root in $(0,1)$ of the equation
	\begin{align*}
		2r^N-p(1-r)=0.
	\end{align*}
	The constant $R^N_{p}$ is best possible.
\end{cor}

As a consequence of Corollary \ref{Th-3.3} (i.e. in particular cases $ p=1 $, $ N=1 $
or $p=2$, $N=1$) we obtain the following results which are sharp refined versions of the Bohr inequality for complex Banach spaces.

\begin{cor}
	Let $f : {B}_X\to\mathbb{D}$ be a holomorphic function with the homogeneous polynomial expansion $f(z)=\sum_{s=0}^{\infty}P_{s}(z)$, $z\in {B}_X$. Then we have
	\begin{align*}
		\sum_{s=0}^{\infty}|P_{s}(z)|+\left(\frac{1}{1+|f(0)|}+\frac{r}{1-r}\right)\sum_{s=1}^{\infty}|P_{s}(z)|^2\leq 1\;\;\mbox{for}\;\; ||z||=r\leq \dfrac{1}{3}.
	\end{align*}
	The constant $1/3$ is best possible. Moreover,
	\begin{align*}
		|f(0)|^2+\sum_{s=1}^{\infty}|P_{s}(z)|+\left(\frac{1}{1+|f(0)|}+\frac{r}{1-r}\right)\sum_{s=1}^{\infty}|P_{s}(z)|^2\leq 1\;\;\mbox{for}\;\; ||z||=r\leq \dfrac{1}{2}.
	\end{align*}
	The constant $1/2$ is best possible.
\end{cor}

 It is a natural investigation whether the improved Bohr inequality 
 using a quantity similar to $ S_r $ given by (\ref{Sr})
 can be established for complex Banach spaces.
The motivation for the study on the Bohr phenomenon via proper combinations in complex Banach spaces is based on the discussions above and the observation of the following remark concerning an improved Bohr inequality using the quantity $ S_r $.

\begin{rem}\label{Rem-2.1}
	Ismagilov \emph{et al.} \cite{Ismagilov- Kayumov- Ponnusamy-2020-JMAA} remarked that for any function $F : [0, \infty)\to [0, \infty)$ such that $F(t)>0$ for $t>0$, there exists an analytic function $f : \mathbb{D}\to\mathbb{D}$ for which the inequality
	\begin{align*}
		\sum_{s=0}^{\infty}|a_s|r^s+\frac{16}{9}\left(\frac{S_r}{\pi}\right)+\lambda\left(\frac{S_r}{\pi}\right)^2+F(S_r)\leq 1\;\; \mbox{for}\;\; r\leq\frac{1}{3}
	\end{align*}
	is false, where $ S_r $ is given by (\ref{Sr}) and $\lambda$ is given in \cite[Theorem 1]{Ismagilov- Kayumov- Ponnusamy-2020-JMAA}.
\end{rem}

Considering Remark \ref{Rem-2.1}, it is noteworthy to observe that augmenting a non-negative quantity with the Bohr inequality does not yield the desired inequality for the class $\mathcal{H}\left(\mathbb{D},\mathbb{D}\right)$. This observation motivates us to study the Bohr inequality further for complex Banach spaces with a suitable setting.\vspace{1.2mm}

Similar to the quantity $S_r$ for functions $f\in \mathcal{H}\left(\mathbb{D},\mathbb{D}\right)$, we define $S_z^*$ for holomorphic functions $f : {B}_X\to\mathbb{D}$ with the homogeneous polynomial expansion $f(z)=\sum_{s=0}^{\infty}P_{s}(z)$, $z\in {B}_X$ given by 
\[
	S_z^*:=\sum_{s=1}^{\infty}s |P_{s}(z)|^2.
\]
To serve our purpose, we consider a polynomial of degree $ m $ as follows
\begin{align}\label{BS-eq-2.77}
	G_m(t):=d_1t+d_2t^2+\dots+d_mt^m, \;\; \mbox{where} \;\;
	d_i\geq 0, i=1,2,\dotsm m
\end{align}
and obtain the following new  improved version of the Bohr inequality for holomorphic functions of the form $f(z)=\sum_{s=0}^{\infty}P_{s}(z)$, $z\in {B}_X$.
\begin{thm}\label{Th-3.4}
	Let $f : {B}_X\to\mathbb{D}$ be a holomorphic function with the homogeneous polynomial expansion $f(z)=\sum_{s=0}^{\infty}P_{s}(z)$, $z\in {B}_X$
	and let $G_m$ be as in \eqref{BS-eq-2.77}. We have the following sharp inequality:
	\begin{align*}
		\mathcal{I}^{f}_{m}(z):=\sum_{s=0}^{\infty}|P_{s}(z)|+\left(\frac{1}{1+|f(0)|}+\frac{r}{1-r}\right)\sum_{s=1}^{\infty}|P_{s}(z)|^2+G_m\left( S_z^*\right)\leq 1\
	\end{align*}
	for $||z||=r\leq 1/3,$
	where the coefficients of the polynomial $G_m$ satisfy the condition
	\begin{align}\label{BS-eq-2.3}
		8d_1\left(\dfrac{3}{8}\right)^2+6c_2d_2\left(\dfrac{3}{8}\right)^4+\dots+2(2m-1)c_md_m\left(\dfrac{3}{8}\right)^{2m}\leq 1,
	\end{align}
	with
	\[
	c_s:=\max_{a\in [0,1]}\left(a(1+a)^2(1-a^2)^{2s-2}\right),
	\quad s=2,3,\cdots, m.
	\]
	The constant $ 1/3$ is best possible
	for each $d_1,\dots, d_m$ which satisfy $(\ref{BS-eq-2.3})$.
\end{thm}

We have the following immediate corollary of Theorem \ref{Th-3.4}
(cf. \cite[Theorem 4]{Liu-Liu-Ponnusamy-BSM-2021} in the case $B_X=\mathbb{D}$).

\begin{cor} Let $f : {B}_X\to\mathbb{D}$ be a holomorphic function with the homogeneous polynomial expansion $f(z)=\sum_{s=0}^{\infty}P_{s}(z)$, $z\in {B}_X$.  Then we have
	\begin{align*}
		\sum_{s=0}^{\infty}|P_{s}(z)|+\left(\frac{1}{1+|f(0)|}+\frac{r}{1-r}\right)\sum_{s=1}^{\infty}|P_{s}(z)|^2+\dfrac{8}{9}S_z^*\leq 1\;\;\mbox{for}\;\; ||z||=r\leq \dfrac{1}{3}.
	\end{align*}
	The constant $ 1/3$ is best possible.
\end{cor}

\section{Auxiliary results}
In this section, we present some necessary lemmas which will be used in proving our main results. The first lemma is about a refined Bohr inequality for the class $\mathcal{H}\left(\mathbb{D},\mathbb{D}\right)$. 
\begin{lemA}\cite[Lemma 4]{Liu-Liu-Ponnusamy-BSM-2021}
	Suppose that $ f(z)=\sum_{s=0}^{\infty}a_sz^s\in \mathcal{H}\left(\mathbb{D},\mathbb{D}\right)$. Then for any $N\in\mathbb{N}$, the following inequality holds:
	\begin{align*}
		\sum_{s=N}^{\infty}|a_s|r^s &+\mathrm{sgn}(t)  \sum_{s=1}^{t}|a_s|^2\dfrac{r^N}{1-r}+\left(\dfrac{1}{1+|a_0|}+\dfrac{r}{1-r}\right)\sum_{s=t+1}^{\infty}|a_s|^2r^{2s}
		\\
		&\leq \dfrac{(1-|a_0|^2)r^N}{1-r}, 
	\end{align*}
	for $|z|= r\in[0,1),$ where $t=\lfloor{(N-1)/2}\rfloor$.
\end{lemA}
In the next lemma, we show the existence of a root in $(0,1)$ of equation \eqref{eq-1-2-b}.
\begin{lem}
	\label{lem-solution}
	Let $r_{p,m}^{***}$ be the maximal positive root of the equation
	\eqref{eq-1-2-b} in $(0,1)$.
	Then $r_{p,0}^{***}=1/\sqrt[p]{3}$ and $r_{p,m}^{***} \in (1/\sqrt[p]{3}, 1)$
	for $1\leq m\leq p$.
\end{lem}
\begin{proof}[\bf Proof of Lemma \ref{lem-solution}]
	First, assume that $m=0$.
	Then $r_{p,0}^{***}$ is the maximal positive root of the equation
	$5r^{2p}-2r^{p}+1+4r^{2p}-4r^{p}=9r^{2p}-6r^{p}+1=0$
	in $(0,1)$.
	Therefore, we have
	$r_{p,0}^{***}=1/\sqrt[p]{3}$.
	
	Next, assume that
	$1\leq m\leq p$. 
	Let $\varphi(r)=5r^{2p+m}-2r^{p+m}+r^{m}+4r^{2p}-4r^{p}$
	and let $r_{0}=1/\sqrt[p]{3}$.
	Then
	\begin{align*}
		\varphi(r_{0})&=
		\frac{5}{9}r_{0}^{m}-\frac{2}{3}r_{0}^{m}+r_{0}^{m}+\frac{4}{9}-\frac{4}{3}
		\\&=
		\frac{8}{9}(r_{0}^m-1)<0.
	\end{align*}
	Since $\varphi(1)=4>0$, this implies that 
	$r_{p,m}^{***} \in (1/\sqrt[p]{3}, 1)$
	for $1\leq m\leq p$.
	This completes the proof.
\end{proof}

We obtain the following lemma
for holomorphic functions $f :{B}_X\to\mathbb{D}$ with the homogeneous polynomial expansion $f(z)=\sum_{s=0}^{\infty}P_{s}(z)$, $z\in {B}_X$,
which extends Lemma A.

\begin{lem}\label{Lem-2.3}
	Let $f :{B}_X\to\mathbb{D}$ be a holomorphic function with the homogeneous polynomial expansion $f(z)=\sum_{s=0}^{\infty}P_{s}(z)$, $z\in {B}_X$. Then for any $N\in \mathbb{N},$ the following inequality holds:
	\begin{align*}
		&\sum_{s=N}^{\infty}|P_s(z)|+\mathrm{sgn}(t)\sum_{s=1}^{t}|P_s(z)|^2\frac{r^{N-2s}}{1-r}+\left(\frac{1}{1+|f(0)|}+\frac{r}{1-r}\right)\sum_{s=t+1}^{\infty}|P_s(z)|^2\\&\leq \left(1-|f(0)|^2\right)\frac{r^N}{1-r}.\nonumber
	\end{align*}
	for $||z||= r\in[0,1),$ where $t=\lfloor{(N-1)/2}\rfloor.$
\end{lem}
\begin{proof}[\bf Proof of Lemma \ref{Lem-2.3}]
	Let $z\in {B}_X\setminus\{0\}$ be fixed and let $\omega=z/||z||$. Let 
	\begin{align*}
		F(\lambda)=f(\lambda\omega),\;\;\lambda\in\mathbb{D}.
	\end{align*}
	
	Then $F : \mathbb{D}\to \mathbb{D}$ is holomorphic and $F(\lambda)=\sum_{s=0}^{\infty}P_s(\omega)\lambda^s$, $\lambda\in\mathbb{D}$.
	Applying Lemma A to the function $F$, we obtain
	\begin{align*}
		\sum_{s=N}^{\infty}|P_s(\omega)||\lambda|^s&+\mathrm{sgn}(t)\sum_{s=1}^{t}|P_s(\omega)|^2\frac{|\lambda|^N}{1-|\lambda|}\\
		&\qquad +\left(\frac{1}{1+|f(0)|}+\frac{|\lambda|}{1-|\lambda|}\right)\sum_{s=t+1}^{\infty}|P_s(\omega)|^2|\lambda|^{2s}\\&\leq \left(1-|F(0)|^2\right)\frac{|\lambda|^N}{1-|\lambda|}.
	\end{align*}
	Setting $\lambda=||z||=r<1$ (note that $|F(0)|=|f(0)|$), we have 
	\begin{align*}
		&\sum_{s=N}^{\infty}|P_s(z)|+\mathrm{sgn}(t)\sum_{s=1}^{t}|P_s(z)|^2\frac{r^{N-2s}}{1-r}+\left(\frac{1}{1+|f(0)|}+\frac{r}{1-r}\right)\sum_{s=t+1}^{\infty}|P_s(z)|^2\\&\leq \left(1-|f(0)|^2\right)\frac{r^N}{1-r}.
	\end{align*}
	This completes the proof.
\end{proof}

For $\omega\in\partial{B}_X$, $a\in (0, 1)$ and $T_{\omega}\in T(\omega)$,
let
\begin{align}\label{Eq-2.2}
		f_a(z)=\frac{a+T_{\omega}(z)}{1+aT_{\omega}(z)},
		\quad
		 z\in {B}_X.
\end{align}
For $z=r\omega$, $|f_a(0)|=a,$ by a simple computation  it follows that 
	\begin{align}\label{EQ-4.2}
		|P_s(z)|=(1-{a}^2){a}^{s-1}r^s\;\;\mbox{for}\;\; s\in \mathbb{N}.
	\end{align} 
	
\section{Proof of the main results}
\begin{proof}[\bf Proof of Theorem \ref{Th-3.22}]
	Let $z\in {B}_X\setminus\{0\}$ be fixed and $\omega=z/||z||$. Let \begin{align*}
		F(\lambda)=f(\lambda\omega),\;\;\lambda\in\mathbb{D}.
	\end{align*}
	Then, we see that $F : \mathbb{D}\to \overline{\mathbb{D}}$ is holomorphic and 
	\begin{align*}
		F(\lambda)=\lambda^m\sum_{s=0}^{\infty}P_{ps+m}(\omega)\lambda^{ps}=\lambda^mg(\lambda^{p}),\; \lambda\in\mathbb{D},
	\end{align*}
	where $g(\lambda):=\sum_{s=0}^{\infty}P_{ps+m}(\omega)\lambda^{s}$. 
	It is clear that $g$ is a holomorphic function on $\mathbb{D}$
	with $g(\mathbb{D})\subset \overline{\mathbb{D}}$.
	If $g(\mathbb{D})\cap \partial \mathbb{D}\neq \emptyset$,
then by the maximal principle for holomorphic functions, $g$ is a constant function on $\mathbb{D}$
	and thus $|P_{m}(\omega)|=1$ and $P_{ps+m}(\omega)=0$ for $s\geq 1$. Thus, we see that $\mathcal{A}^f_{p, m}(z)\leq 1$ for all $z=r\omega$ with $r\in [0,1)$.\vspace{2mm}
	
	Next, we assume that $g(\mathbb{D})\subset \mathbb{D}$.
	In this case, applying Lemma A with $N=1$ to the function $g$, we obtain 
	\begin{align*}
		&\sum_{s=0}^{\infty}|P_{ps+m}(\omega)||\lambda|^{ps}+\left(\frac{1}{1+|P_m(\omega)|}+\frac{|\lambda|^p}{1-|\lambda|^p}\right)\sum_{s=1}^{\infty}|P_{ps+m}(\omega)|^2|\lambda|^{2ps}\\&\leq |P_m(\omega)|+\left(1-|P_m(\omega)|^2\right)\frac{|\lambda|^p}{1-|\lambda|^p}.
	\end{align*}
	Multiplying $ |\lambda|^m$ on both sides, we have 
	\begin{align*}
		&\sum_{s=0}^{\infty}|P_{ps+m}(\omega)||\lambda|^{ps+m}+\left(\frac{1}{|\lambda|^{m}+|P_m(\omega)||\lambda|^{m}}+\frac{|\lambda|^{p-m}}{1-|\lambda|^p}\right)\sum_{s=1}^{\infty}|P_{ps+m}(\omega)|^2|\lambda|^{2ps+2m}\\&\leq |P_m(\omega)||\lambda|^m+\left(1-|P_m(\omega)|^2\right)\frac{|\lambda|^{p+m}}{1-|\lambda|^p}.
	\end{align*}
	Let $a=|P_m(\omega)|\in [0,1)$ and $\lambda=||z||=r$.
	Then, a simple computation shows that
	\begin{align*}
		\mathcal{A}^f_{p, m}(r\omega)&\leq r^m\left( a+ (1-a^2)\frac{r^p}{1-r^p}\right)
		\\
		&=
		-\frac{r^{p+m}}{1-r^{p}}\left(a-\frac{1-r^p}{2r^p}\right)^2+1+\frac{\varphi(r)}{4r^p(1-r^p)},
	\end{align*}
	where $\varphi(r)=5r^{2p+m}-2r^{p+m}+r^{m}+4r^{2p}-4r^{p}$.
	Therefore, by virtue of Lemma \ref{lem-solution}
	and the fact that the function
	\begin{align*}
	r^m\left( a+ (1-a^2)\frac{r^p}{1-r^p}\right)
	\end{align*}
	is increasing for $r\in (0,1)$,
	the desired inequality $\mathcal{A}^f_{p, m}(z)\leq 1$ is established for $r\leq r_{p,m}^{***}$.\vspace{2mm}
	
	To prove the sharpness, we consider the function 
	\begin{align*}
		f_{a,p,m}(z)=T_{\omega}^m(z)\left(\frac{T_{\omega}^p(z)-a}{1-aT_{\omega}^p(z)}\right),
	\end{align*}
	where $z\in {B}_X$, $\omega\in\partial {B}_X,$ $a\in (0, 1)$ and $T_{\omega}\in T(\omega)$. For $f_{a,p,m}$ and $z=r\omega$, we obtain $|P_{sp+m}(z)|=(1-a^2)a^{s-1}r^{sp+m}$, $s\in\mathbb{N}$, $|P_m(\omega)|=a$. Thus, we have 
	\begin{align*}
		\mathcal{A}^{f_{a,p,m}}_{p, m}(z)&=r^m\Bigg(a+(1-a^2)\sum_{s=1}^{\infty}a^{s-1}r^{sp}\\&\qquad +\left(\frac{r^{-m}}{1+a}+\frac{r^{p-m}}{1-r^p}\right)\sum_{s=1}^{\infty}(1-a^2)^2a^{2s-2}r^{2sp+m}\Bigg)\\&=r^m\left(a+(1-a^2)\frac{r^{p}}{1-r^p}\right)
	\end{align*}
	which is bigger than $1$ if $r>r_{p,m}^{***}$ and $a=(1-(r_{p,m}^{***})^{p})/(2(r_{p,m}^{***})^p)\in (0,1)$
	in the case $1\leq m\leq p$. 
	
	In the case $m=0$,
	for $r\in (r_{p,0}^{***},1)$,
	let $c=r^p/(1-r^p)$.
	Then we have $c>1/2$.
	Then by taking $a=1/(2c)$,
	we have
\begin{align*}
	\mathcal{A}^{f_{a,p,m}}_{p, m}(z)=c+\frac{1}{4c}>2\sqrt{c\cdot\frac{1}{4c}}=1.
\end{align*}
	This completes the proof.
\end{proof}

\begin{proof}[\bf Proof of Theorem \ref{BS-thm-2.8}]
	Define $F(z)=T_v\left(f(z)\right)$, $z\in B_X$, where $v\in Y$ is fixed with $||v||=1$. 
	Then it is easy to see that $F\in\mathcal{H}\left(B_X, \overline{\mathbb{D}}\right)$ and according to the hypothesis of Theorem  \ref{BS-thm-2.8}, the homogeneous polynomial expansion of $F$ is as follows.
	\begin{align*}
	F(z)=\sum_{s=0}^{\infty}P_{sp+m}(z),
	\quad z\in B_X,
	\end{align*}
	where
	\begin{align*}
		P_{sp+m}(z)=\frac{T_v\left(D^{sp+m}f(0)\left(z^{sp+m}\right)\right)}{(sp+m)!}.
	\end{align*}
	By Theorem \ref{Th-3.22},
	we conclude that $\mathcal{B}_{p,m}^f(z)\leq 1$ for $r\leq r_{p,m}^{***}$,
	where ${r}_{p,m}^{***}$ is the maximal positive root of the equation
\eqref{eq-1-2-b}
	in $(0,1)$.
	
	In order to show the sharpness, 
we consider the mapping $\tilde{f}_{a,p,m}$ given by 
	\begin{align*}
		\tilde{f}_{a,p,m}(z)=T_u^{m}(z)\frac{T_u(z)^{p}-a}{1-aT_u(z)^{p}}v,\; z\in B_X
	\end{align*}
	for some $a\in [0, 1)$, where $u\in X$ is fixed with $||u||=1$ and $v\in Y$ is fixed with $||v||=1$.
	As in the proof of Theorem \ref{Th-3.22},
	it can be proved that the function $\tilde{f}_{a,p,m}$ gives the sharpness of $ r_{p,m}^{***}$.
	This completes the proof.
\end{proof}

\begin{proof}[\bf Proof of Theorem \ref{BS-thm-2.7}]
	Let $u\in B_X\setminus\{0\}$ be fixed. 
	Define $F(z)=T_{u}\left(f(z)\right)$, $z\in B_X$. 
	Then $F\in \mathcal{H}\left(B_X,\overline{\mathbb{D}}\right)$ and
	according to the hypothesis of Theorem \ref{BS-thm-2.7}, the homogeneous polynomial expansion  of $F$ is as follows:
	\begin{align*}
		F(z)=\left(\frac{D^{m-1}g(0)\left(z^{m-1}\right)}{(m-1)!}+\sum_{s=1}^{\infty}\frac{D^{sp+m-1}g(0)\left(z^{sp+m-1}\right)}{(sp+m-1)!}\right) T_u(z).
	\end{align*}
	Note that 
	\[
	P_{sp+m}(z)=\frac{D^{sp+m-1}g(0)\left(z^{sp+m-1}\right)}{(sp+m-1)!} T_u(z)
	=T_u\left(\frac{D^{sp+m}f(0)(z^{sp+m})}{(sp+m)!}\right)
	\]
	is a homogeneous polynomial of degree $sp+m$.
	Applying Theorem \ref{Th-3.22} for $z=u$, we have
	\begin{align*}
		\sum_{s=0}^{\infty}|P_{sp+m}(u)|+\left(\frac{1}{r^m+|P_m(u)|}+\frac{r^{p-m}}{1-r^p}\right)\sum_{s=1}^{\infty}|P_{sp+m}(u)|^2\leq 1
	\end{align*}
	for $||u||=r\leq r_{p,m}^{***}$, where ${r}_{p,m}^{***}$ is the maximal positive root of the equation
\eqref{eq-1-2-b}
	in $(0,1)$.
	Since
	\begin{align*}
		\frac{||D^{sp+m}f(0)(u^{sp+m})||}{(sp+m)!}
		&= \left|\frac{D^{sp+m-1}g(0)\left(u^{sp+m-1}\right)}{(sp+m-1)!}\right|\cdot \| u\|
		\\
		&=\left|\frac{D^{sp+m-1}g(0)\left(u^{sp+m-1}\right)}{(sp+m-1)!}\right|\cdot T_u(u)
		\\
		&= |P_{sp+m}(u)|,
	\end{align*}
	the desired inequality $\mathcal{C}^{f}_{p, m}(z)\leq 1$ is established for $r\leq r_{p,m}^{***}$.\vspace{1.2mm}
	
	The next step is to show that the constant $r_{p,m}^{***}$ is sharp
	 for $p,m\in \mathbb{N}$ with $1\leq m\leq p$. To serve the purpose, we consider the function $\hat{f}_{a,p,m}$ given by 
	\begin{align*}
		\hat{f}_{a,p,m}(z)=zT_v^{m-1}(z)\frac{T_v^p(z)-a}{1-aT_v^p(z)},\; z\in B_X
	\end{align*}
	for some $a\in [0, 1)$, where $v\in X$ is fixed with $||v||=1$. Putting $z=rv$, it follows that 
	\begin{align*}
		\frac{||D^m\hat{f}_{a,p,m}(0)\left(z^m\right)||}{m!}=ar^m\; \mbox{and}\; \frac{||D^{sp+m}\hat{f}_{a,p,m}(0)\left(z^{sp+m}\right)||}{(sp+m)!}=\left(1-a^2\right)a^{s-1}r^{sp+m}.
	\end{align*}
	For $\hat{f}_{a,p,m}$, a straightforward computation shows that
	\begin{align*}
		\mathcal{C}^{\hat{f}_{a,p,m}}_{p, m}(rv)&=r^m\Bigg(a+(1-a^2)\sum_{s=1}^{\infty}a^{s-1}r^{sp}\\&\qquad +\left(\frac{r^{-m}}{1+a}+\frac{r^{p-m}}{1-r^p}\right)\sum_{s=1}^{\infty}(1-a^2)^2a^{2s-2}r^{2sp+m}\Bigg)\\&=r^m\left(a+(1-a^2)\frac{r^{p}}{1-r^p}\right).
	\end{align*}
	Then as in the proof of Theorem \ref{Th-3.22},
	for  $r>r_{p,m}^{***}$, there exists  $a\in (0,1)$ such that
$\mathcal{C}^{\hat{f}_{a,p,m}}_{p, m}(rv)>1$.
This completes the proof.	
\end{proof}

\begin{proof}[\bf Proof of Theorem \ref{BS-thm-2.6a}]
	Let $z\in {B}_X\setminus\{0\}$ be fixed and  $z_0=z/||z||$. Define $h(\lambda)=f(\lambda z_0)$ for $\lambda\in\mathbb{D}$. Then $h\in\mathcal{H}(\mathbb{D}, \overline{\mathbb{D}})$ and according to the hypothesis of Theorem \ref{BS-thm-2.6a}, the series expansion of $h$ is as follows
	\begin{align*}
		h(\lambda)=P_m(z_0)\lambda^m+\sum_{s=N}^{\infty}P_s(z_0)\lambda^s.
	\end{align*}
	Let
	\begin{align*}
		\omega(\lambda)=a_m+\sum_{s=N}^{\infty}a_s\lambda^{s-m},
		\quad
		\lambda \in \mathbb{D},
	\end{align*}
	where $a_m=P_m(z_0)$ {and} $a_s=P_s(z_0)$, $s\geq N$. 
	Then we conclude that $h(\lambda)=\lambda^m \omega(\lambda)$ and 
	$\omega\in\mathcal{H}(\mathbb{D}, \overline{\mathbb{D}})$ due to $h\in\mathcal{H}(\mathbb{D}, \overline{\mathbb{D}})$.
	If $\omega(\mathbb{D})\cap \partial \mathbb{D}\neq \emptyset$, then by the maximum principle for holomorphic functions,
	there exists $\theta \in \mathbb{R}$ such that $\omega(\lambda)\equiv e^{i\theta}$.
	Then $h(\lambda)=e^{i\theta}\lambda^m$
	and $\mathcal{D}_{N,m}^f(rz_0)\leq 1$ for all $r<1$.
	
	So, we may assume that $\omega \in\mathcal{H}(\mathbb{D}, {\mathbb{D}})$.
	Applying Lemma A to $\omega$ with $N-m$, we obtain 
	\begin{align*}
		\sum_{s=N-m}^{\infty}&|a_{s+m}||\lambda|^s+\mathrm{sgn}(t)\sum_{s=1}^{t}|a_{s+m}|^2\frac{|\lambda|^{N-m}}{1-|\lambda|}
		\\&\qquad +\left(\frac{1}{1+|a_m|}+\frac{|\lambda|}{1-|\lambda|}\right)\sum_{s=t+1}^{\infty}|a_{s+m}|^2|\lambda|^{2s}\\&\leq \left(1-|a_m|^2\right)\frac{|\lambda|^{N-m}}{1-|\lambda|},
	\end{align*}
	where $t=\lfloor{(N-m-1)/2}\rfloor$.
	If $t\geq 1$, then we have $m<t+m<(N-m-1)+m=N-1$.
	So, the above inequality becomes
	\begin{align*}
		\sum_{s=N-m}^{\infty}&|a_{s+m}||\lambda|^s+\left(\frac{1}{1+|a_m|}+\frac{|\lambda|}{1-|\lambda|}\right)\sum_{s=N-m}^{\infty}|a_{s+m}|^2|\lambda|^{2s}\\&\leq \left(1-|a_m|^2\right)\frac{|\lambda|^{N-m}}{1-|\lambda|},
	\end{align*}
	Thus, multiplying by $|\lambda|^m$,
	the following estimate can be established.
	\begin{align*}
		\sum_{s=N}^{\infty}&|P_s(z_0)|\cdot |\lambda|^s
		+\left(\frac{|\lambda|^m}{|\lambda|^m+|P_m(z_0)|\cdot|\lambda|^m}+\frac{|\lambda|}{1-|\lambda|}\right)\sum_{s=N}^{\infty}|P_s(z_0)|^2|\lambda|^{2s-m}\\&\leq \left(1-|a_m|^2\right)\frac{|\lambda|^N}{1-|\lambda|}.
	\end{align*}
	Setting $\lambda=||z||=r$ and $z=\lambda z_0$, we obtain 
	\begin{align*}
		\mathcal{D}_{N,m}^f(z)\leq |a_m|r^m+\left(1-|a_m|^2\right)\frac{r^N}{1-r}:=\mathcal{J}^f_{N,m}(|a_m|,r).
	\end{align*}

To establish the desired inequality $\mathcal{D}_{N,m}^f(z)\leq 1$ for $||z||=r\leq r_{N, m}^{**}$, it suffices to demonstrate that the inequality $\mathcal{J}^f_{N,m}(|a_m|,r)\leq 1$ is valid for $r\leq r_{N, m}^{**}$.

Since
\begin{align*}
	\mathcal{J}^f_{N,m}(|a_m|,r)&=1-\frac{r^N}{(1-r)}\left(|a_m|-\frac{1-r}{2r^{N-m}}\right)^2\\&\quad+\frac{4r^{2N-m}+4r^{N+1-m}-4r^{N-m}+r^{m+2}-2r^{m+1}+r^m}{4(1-r)r^{N-m}},
\end{align*}
we have
\begin{align*}
\mathcal{J}^f_{N,m}(|a_m|, {r}_{N,m}^{**})&\leq 1,
\end{align*}
which combined with
the fact that
the function $\mathcal{J}^f_{N,m}(|a_m|,r)$ is increasing for $r\in (0,1)$
implies that 
$\mathcal{J}^f_{N,m}(|a_m|, {r})\leq 1$ for $r\in (0, {r}_{N,m}^{**})$.

Therefore, we conclude that $\mathcal{D}_{N,m}^f(z)\leq 1$ for $r\leq r_{N,m}^{**}$,
where ${r}_{N,m}^{**}$ is the maximal positive root of the equation
\eqref{r*}. 
Thus, the inequality is established.\vspace{1.2mm}

In order to show the sharpness, 
we consider the function $f_{a,m}$ given by 
	\begin{align*}
		f_{a,m}(z)=T_v^{m}(z)\frac{T_v(z)-a}{1-aT_v(z)},\; z\in B_X
	\end{align*}
	for some $a\in [0, 1)$, $v\in X$ is fixed with $||v||=1$. Putting $z=rv$, it follows that 
	\begin{align*}
		\frac{|D^mf_{a,m}(0)\left(z^m\right)|}{m!}
		=ar^m\; \mbox{and}\; \frac{|D^{s+m}f_{a,m}(0)\left(z^{s+m}\right)|}{(s+m)!}=\left(1-a^2\right)a^{s-1}r^{s+m}.
	\end{align*}
	For $f_{a,m}$, a straightforward computation shows that
	\begin{align*}
		\mathcal{D}_{m+1,m}^{f_{a,m}}(rv)&=r^m\Bigg(a+(1-a^2)\sum_{s=1}^{\infty}a^{s-1}r^{s}\\&\qquad +\left(\frac{r^{-m}}{1+a}+\frac{r^{1-m}}{1-r}\right)\sum_{s=1}^{\infty}(1-a^2)^2a^{2s-2}r^{2s+m}\Bigg)\\&=r^m\left(a+(1-a^2)\frac{r}{1-r}\right),
	\end{align*}
	which is bigger than $1$ if, $r>r_{m+1,m}^{**}$ and $a=(1-r_{m+1,m}^{**})/(2r_{m+1,m}^{**})\in (0,1)$
	in the case $1\leq m$. 
	
	In the case $m=0$,
	for $r\in (r_{1,0}^{**},1)$,
	let $c=r/(1-r)$.
	Then we have $c>1/2$.
	Then by taking $a=1/(2c)$,
	we have
\begin{align*}
	\mathcal{D}_{1,0}^{f_{a,0}}(z)=c+\frac{1}{4c}>2\sqrt{c\cdot\frac{1}{4c}}=1.
\end{align*} 
This completes the proof. 
\end{proof}

\begin{proof}[\bf Proof of Theorem \ref{BS-thm-2.6}]
Define $F(z)=T_v\left(f(z)\right)$, $z\in B_X$, where $v\in Y$ is fixed with $||v||=1$. 
	Then it is easy to see that $F\in\mathcal{H}\left(B_X, \overline{\mathbb{D}}\right)$ and according to the hypothesis of Theorem  \ref{BS-thm-2.6}, the homogeneous polynomial expansion of $F$ is as follows.
	\begin{align*}
	F(z)=P_m(z)+\sum_{s=N}^{\infty}P_{s}(z),
	\quad z\in B_X,
	\end{align*}
	where
	\begin{align*}
		P_{s}(z)=\frac{T_v\left(D^{s}f(0)\left(z^{s}\right)\right)}{s!},
		\quad s=m, \mbox{or } s\geq N.
	\end{align*}
	By Theorem \ref{BS-thm-2.6a},
	we conclude that $\mathcal{E}_{p,m}^f(z)\leq 1$ for $r\leq r_{N,m}^{**}$,
where ${r}_{N,m}^{**}$ is the maximal positive root of the equation
\eqref{r*}. 
	
	In order to show the sharpness, 
we consider the mapping $\tilde{f}_{a,m}$ given by 
	\begin{align*}
		\tilde{f}_{a,m}(z)=T_u^{m}(z)\frac{T_u(z)-a}{1-aT_u(z)}v,\; z\in B_X
	\end{align*}
	for some $a\in [0, 1)$, $u\in X$ is fixed with $||u||=1$ and $v\in Y$ is fixed with $||v||=1$.
	As in the proof of Theorem \ref{BS-thm-2.6a},
	it can be proved that the function $\tilde{f}_{a,m}$ gives the sharpness of $ r_{m+1,m}^{**}$.
	This completes the proof.
\end{proof}

\begin{proof}[\bf Proof of Theorem \ref{BS-thm-2.5}]
Let $u\in B_X\setminus\{0\}$ be fixed. 
	Define $F(z)=T_{u}\left(f(z)\right)$, $z\in B_X$. 
	Then $F\in \mathcal{H}\left(B_X,\overline{\mathbb{D}}\right)$ and
	according to the hypothesis of Theorem \ref{BS-thm-2.5}, the homogeneous polynomial expansion  of $F$ is as follows:
	\begin{align*}
		F(z)=\left(\frac{D^{m-1}g(0)\left(z^{m-1}\right)}{(m-1)!}+\sum_{s=N}^{\infty}\frac{D^{s-1}g(0)\left(z^{s-1}\right)}{(s-1)!}\right) T_u(z).
	\end{align*}
	Note that 
	\[
	P_{s}(z)=\frac{D^{s-1}g(0)\left(z^{s-1}\right)}{(s-1)!} T_u(z)
	=T_u\left(\frac{D^{s}f(0)(z^{s})}{s!}\right)
	\]
	is a homogeneous polynomial of degree $s$.
	Applying Theorem \ref{BS-thm-2.6a} for $z=u$, we have
	\begin{align*}
	|P_m(u)|+\sum_{s=N}^{\infty}|P_s(u)|
	+\left(\frac{1}{||u||^m+|P_m(u)|}+\frac{||u||^{1-m}}{1-||u||}\right)\sum_{s=N}^{\infty}|P_s(u)|^2\leq 1
	\end{align*}
	for $||u||=r\leq r_{N,m}^{**}$, where ${r}_{N,m}^{**}$ is the maximal positive root of the equation
\eqref{r*}.
	Since
	\begin{align*}
		\frac{||D^{s}f(0)(u^{s})||}{s!}
		= |P_{s}(u)|,
	\end{align*}
	the desired inequality $\mathcal{F}^{f}_{p, m}(z)\leq 1$ is established for $r\leq r_{p,m}^{**}$.\vspace{1.2mm}
	
	The next step is to show that the constant $r_{m+1,m}^{**}$ is sharp. To serve the purpose, we consider the function 
	$\hat{f}_{a,m}$ given by 
	\begin{align*}
		\hat{f}_{a,m}(z)=zT_v^{m-1}(z)\frac{T_v(z)-a}{1-aT_v(z)},\; z\in B_X
	\end{align*}
	for some $a\in [0, 1)$, where $v\in X$ is fixed with $||v||=1$. Putting $z=rv$, it follows that 
	\begin{align*}
		\frac{||D^{s}\hat{f}_{a,m}(0)\left(z^{s}\right)||}{s!}=\left(1-a^2\right)a^{s-1}r^{s+m}.
	\end{align*}
	For $\hat{f}_{a,m}$, a straightforward computation shows that
	\begin{align*}
		\mathcal{F}^{\hat{f}_{a,m}}_{m+1, m}(rv)&=r^m\Bigg(a+(1-a^2)\sum_{s=1}^{\infty}a^{s-1}r^{s}\\&\qquad +\left(\frac{r^{-m}}{1+a}+\frac{r^{1-m}}{1-r}\right)\sum_{s=1}^{\infty}(1-a^2)^2a^{2s-2}r^{2s+m}\Bigg)\\&=r^m\left(a+(1-a^2)\frac{r}{1-r}\right).
	\end{align*}
	Then as in the proof of Theorem \ref{BS-thm-2.6a},
	for  $r>r_{p,m}^{**}$, there exists  $a\in (0,1)$ such that
$\mathcal{F}^{\hat{f}_{a,m}}_{p, m}(rv)>1$.
This completes the proof.
\end{proof}

\begin{proof}[\bf Proof of Theorem \ref{Th-3.2}]
	Let $z\in {B}_X\setminus\{0\}$ be fixed and let $\omega=z/||z||$. Let 
	\begin{align*}
		F(\lambda)=f(\lambda\omega),\;\;\lambda\in\mathbb{D}.
	\end{align*}
	
	Then $F : \mathbb{D}\to \mathbb{D}$ is holomorphic and $F(\lambda)=\sum_{s=0}^{\infty}P_s(\omega)\lambda^s$, $\lambda\in\mathbb{D}$. By Schwarz-Pick lemma
	for holomorphic functions on $\mathbb{D}$, we have 
	\begin{align*}
		|F(\lambda)|\leq\frac{|F(0)|+|\lambda|}{1+|F(0)||\lambda|}=\frac{|f(0)|+|\lambda|}{1+|f(0)||\lambda|}.
	\end{align*}
	Setting $\lambda=||z||<1$ and $|f(0)|=a$, we have 
	\begin{align}\label{BS-e-3.1}
		|f(z)|\leq \frac{a+\| z\|}{1+a\| z\|}.
	\end{align} 
	In view of the inequalities \eqref{BS-e-3.1} and \eqref{Schwarz-k}, we obtain
	 
	 \begin{align}\label{BS-e-3.11}
	 	|f( v(z))|\leq \frac{a+\| z\|^m}{1+a\| z\|^m}\;\;\mbox{for}\;\;z\in B_X.
	 \end{align}
	 
	By virtue of \eqref{BS-e-3.11} and Lemma \ref{Lem-2.3}, we have 
	\begin{align*}
	\mathcal{G}^{f}_{m, p, N}(z)\leq \left(\frac{a+r^m}{1+ar^m}\right)^p+\left(1-a^2\right)\frac{r^N}{1-r}=1+K_{m,p}(r), 
	\end{align*}
	where
	\begin{align*}
		K_{m,p}(r):=\left(\frac{a+r^m}{1+ar^m}\right)^p+\left(1-a^2\right)\frac{r^N}{1-r}-1.
	\end{align*}
	Taking $\varphi_0(r)=1$ and $N(r)=r^N/(1-r)$ in \cite[Lemma 3]{Chen-Liu-Ponnusamy-RM-2023}, it can be easily shown that $ K_{m,p}(r) \leq 0$ for $r\leq R^N_{p,m}$, where $ R^N_{p,m} $ is the unique positive root of the equation \eqref{BS-ee-3.1}
	in $(0,1)$. 
	Therefore, the desired inequality $\mathcal{G}^{f}_{m, p, N}(z)\leq 1$ holds for $\| z\|=r\leq R^N_{p,m}$.\vspace{2mm} 
	
	We prove the sharpness for $R^N_{p,m}$.
	Let $r>R^N_{p,m}$ be fixed.
	We consider the function $f_a$ given by \eqref{Eq-2.2},
	where $\omega\in\partial {B}_X$, $a\in (0, 1)$ and 
	$T_{\omega}\in T(\omega)$.
	Then for $f_a$,  $v(z)=T_{\omega}(z)^{m-1}z$ and $z=r\omega$, we have 
	\begin{align*}
		\mathcal{G}^{f_a}_{m, p, N}(r\omega)&=
		\left(\frac{a+r^m}{1+ar^m}\right)^p+\sum_{s=N}^{\infty}(1-a^2)a^{s-1}r^s
		\\&\quad
		+\mathrm{sgn}(t)\sum_{s=1}^{t}\left(1-a^2\right)^2a^{2s-2}\frac{r^{N}}{1-r}\\&\quad+\left(\frac{1}{1+a}+\frac{r}{1-r}\right)\left(1-a^2\right)^2\sum_{s=t+1}^{\infty}a^{2s-2}r^{2s}\\&=1+(1-a)\Psi_{p,N,m}(r),
	\end{align*}
	where 
	\begin{align*}
		\Psi_{p,N,m}(r):&=\frac{1}{1-a}\left(\left(\frac{a+r^m}{1+ar^m}\right)^p-1\right)+(1+a)\frac{r^Na^{N-1}}{1-ar}\\&\quad+(1-a^2)(1+a)\mathrm{sgn}(t)\sum_{s=1}^{t}a^{2s-2}\frac{r^{N}}{1-r}\\&\quad+\left(\frac{1}{1+a}+\frac{r}{1-r}\right)\left(1-a^2\right)(1+a)\sum_{s=t+1}^{\infty}a^{2s-2}r^{2s}.
	\end{align*}
	It is clear that $\mathcal{G}^{f_a}_{m, p, N}(r\omega)>1$ if $ \Psi_{p,N,m}(r)> 0$ for some $a\in [0, 1)$. 
Since
	\begin{align*}
		\lim_{a\to 1^{-}}\Psi_{p,N,m}(r)=-p\left(\frac{1-r^m}{1+r^m}\right)+2\frac{r^N}{1-r}>0,
	\end{align*}
by choosing $a$ sufficiently close to $1$, we see that 
$ \Psi_{p,N,m}(r)> 0$.
	This completes the proof.
\end{proof}

\begin{proof}[\bf Proof of Theorem \ref{Th-3.4}]
	Let $z\in {B}_X\setminus\{0\}$ be fixed and let $\omega=z/||z||$ and 
	\begin{align*}
		F(\lambda)=f(\lambda\omega),\;\;\lambda\in\mathbb{D}.
	\end{align*}
	
	It is easy to see that $F\in \mathcal{H}\left(\mathbb{D},\mathbb{D}\right)$ and $F(\lambda)=\sum_{s=0}^{\infty}P_s(\omega)\lambda^s$, $\lambda\in\mathbb{D}$. 
	Moreover, it is well known that  (see e.g. \cite[p. 35]{Graham-2003})
	\begin{align*}
		|P_s(\omega)|\leq 1-|F(0)|^2\;\;\mbox{for}\;\; s\in\mathbb{N}.
	\end{align*}
	Clearly, for $\lambda=||z||=r$ and $a=|F(0)|=|f(0)|,$ we obtain 
	\begin{align}\label{E-2.3}
		|P_s(z)|=|P_s(\omega)||\lambda|^s\leq \left(1-a^2\right)r^s\;\;\mbox{for}\;\; s\in\mathbb{N}.
	\end{align}
	In view of \eqref{E-2.3},  for $||z||=r<1,$ we obtain
	\begin{align}\label{BS-eq-2.7}
		S^*_z=\sum_{s=1}^{\infty}s|P_s(z)|^2\leq  \dfrac{\left(1-{a}^2\right)^2r^2}{(1-r^2)^2}.
	\end{align}
	In view of \eqref{BS-eq-2.77}, \eqref{BS-eq-2.7} and Lemma \ref{Lem-2.3} with $N=1$, by a simple computation, the following inequality can be obtained
	\begin{align*}
		\mathcal{I}^{f}_{m}(z)&\leq a+ \left(1-a^2\right)\frac{r}{1-r}+\sum_{s=1}^{m}d_s\left(\dfrac{(1-a^2)r}{1-r^2}\right)^{2s}\\&= 1+Q(a,r),
	\end{align*}
	where \begin{align*}
		Q(a,r):=\dfrac{(1-a^2)r}{1-r}+\sum_{s=1}^{m}d_s\left(\dfrac{(1-a^2)r}{1-r^2}\right)^{2s}-(1-a).
	\end{align*}
	For all $a\in [0,1)$, by a straightforward calculation, it can be shown that $Q(a,r)$ is a monotonically increasing function of $r$. Consequently, $Q(a,r)\leq Q(a,1/3)$ for $a\in [0,1)$. Upon a straightforward calculation
	\begin{align*}
		Q(a,1/3)=\dfrac{(1-a^2)}{2}\left(1+2F_m(a)-\dfrac{2}{1+a}\right)=\dfrac{(1-a^2)}{2}\Phi(a),
	\end{align*}
	where \begin{align*}
		F_m(a):=\sum_{s=1}^{m}d_s(1-a^2)^{2s-1}\left(\dfrac{3}{8}\right)^{2s}\;\;\mbox{and}\;\; \Phi(a):=1+2F_m(a)-\dfrac{2}{1+a}.
	\end{align*}
	To establish $Q(a,r)\leq 0,$ it suffices to show that $\Phi(a)\leq 0$ for $a\in [0,1].$ As $a\in [0,1 ]$, a straightforward calculation reveals that
	\begin{align*}
		a(1+a)^2\left(\frac{3}{8}\right)^2&\leq 4 \left(\frac{3}{8}\right)^2,\\ a(1+a)^2(1-a^2)^2\left(\dfrac{3}{8}\right)^4&\leq c_2\left(\frac{3}{8}\right)^4,\\ \vdots \\a(1+a)^2(1-a^2)^{2m-2}\left(\dfrac{3}{8}\right)^{2m}&\leq c_m\left(\frac{3}{8}\right)^{2m}.
	\end{align*}
	Thus, we see that 
	\begin{align*}
		\Phi^{\prime}(a)&=\frac{2}{(1+a)^2}\bigg(1-2d_1a(1+a)^2\left(\dfrac{3}{8}\right)^2-6d_2a(1+a)^2(1-a^2)^2\left(\dfrac{3}{8}\right)^4-\cdots\\&\quad-2(2m-1)d_ma(1+a)^2(1-a^2)^{2m-2}\left(\dfrac{3}{8}\right)^{2m}\bigg)\\&\geq \dfrac{2}{(1+a)^2} \left(1-\left(8d_1\left(\dfrac{3}{8}\right)^2+6c_2d_2\left(\dfrac{3}{8}\right)^4+\dots+2(2m-1)c_md_m\left(\dfrac{3}{8}\right)^{2m}\right)\right)\\&\geq 0,
	\end{align*}
	if  the coefficients $d_i$ of the polynomial $G_m$ satisfy the condition outlined in \eqref{BS-eq-2.3}. This indicates that $\Phi(a)$ behaves as an ascending function in $a\in [0,1]$, leading to the conclusion that $\Phi(a)\leq\Phi(1)=0$. This, in turn, establishes the desired inequality.\vspace{1.2mm}
	
	To show the constant $1/3$ is best possible, we consider the function $f_a$ given by \eqref{Eq-2.2},
	where $\omega\in\partial {B}_X$, $a\in (0, 1)$ and $T_{\omega}\in T(\omega)$.
 Using \eqref{EQ-4.2}, it can be readily calculated that
	\begin{align*}
		\mathcal{I}^{f_a}_{m}(r\omega)=1-(1-a)\Psi^*(r),
	\end{align*}
	where 
	\begin{align*}
		\Psi^*(r):=1-\dfrac{(1+a)r}{1-r}-\dfrac{d_1r^2(1-a)(1+a)^2}{(1-a^2r^2)^2}-\dots-\dfrac{d_mr^{2m}(1-a)^{2m-1}(1+a)^{2m}}{(1-a^2r^2)^{2m}}.
	\end{align*}
	For fixed $r>1/3,$ we have 
	\[
	\lim\limits_{a\rightarrow 1^{-}}\Psi^*(r)=\frac{1-3r}{1-r}<0.
	\]
	Thus, it follows that $\Psi^*(r)<0$ for $a$ sufficiently close to $1$. 
	Hence, we have
	\begin{align*}
		\mathcal{I}^{f_a}_{m}(r\omega)=1-(1-a)\Psi^*(r)>1,
	\end{align*}
	which shows that $1/3$ is best possible and thereby concluding the proof.
\end{proof}



\noindent\textbf{Conflict of interest.} The authors declare that there is no conflict  of interest regarding the publication of this paper.\vspace{1.2mm}


\noindent\textbf{Funding.} The first author is supported by SERB, SUR/2022/002244, Govt. India and the second author is supported by UGC-JRF (NTA Ref. No.: $ 201610135853 $), New Delhi, India, and the third author is partially supported by 
JSPS KAKENHI Grant Number JP22K03363.
\vspace{1.5mm}

\noindent\textbf{Author's contribution.} All the authors have contributed equally to preparing the manuscript.

\end{document}